\documentclass[reqno,centertags, 12pt]{amsart}
\usepackage{amsmath,amsthm,amscd,amssymb}
\usepackage{latexsym}
\usepackage{graphicx}
\usepackage[mathscr]{eucal}
\newcommand{\bbC}{{\mathbb{C}}}
\newcommand{\bbD}{{\mathbb{D}}}

\newcommand{\bbR}{{\mathbb{R}}}

\newcommand{\bbZ}{{\mathbb{Z}}}

\newcommand{\fre}{{\frak{e}}}
\newcommand{\frg}{{\frak{g}}}

\newcommand{\x}{{\mathbf{x}}}
\newcommand{\z}{{\mathbf{z}}}

\newcommand{\calC}{{\mathcal{C}}}
\newcommand{\calE}{{\mathcal{E}}}

\newcommand{\calH}{{\mathcal H}}

\newcommand{\calL}{{\mathcal L}}

\newcommand{\lb}{\label}

\newcommand{\ol}{\overline}

\newcommand{\wti}{\widetilde  }

\newcommand{\tr}{\text{\rm{Tr}}}

\newcommand{\supp}{\text{\rm{supp}}}

\newcommand{\bi}{\bibitem}

\newcommand{\beq}{\begin{equation}}
\newcommand{\eeq}{\end{equation}}
\newcommand{\ba}{\begin{align}}
\newcommand{\ea}{\end{align}}

\let\det=\undefined\DeclareMathOperator{\det}{det}

\newcounter{smalllist}
\newenvironment{SL}{\begin{list}{{\rm\roman{smalllist})}}{%
\setlength{\topsep}{0mm}\setlength{\parsep}{0mm}\setlength{\itemsep}{0mm}%
\setlength{\labelwidth}{2em}\setlength{\leftmargin}{2em}\usecounter{smalllist}%
}}{\end{list}}

\newcommand{\comm}[1]{}

\allowdisplaybreaks
\numberwithin{equation}{section}

\newtheorem{theorem}{Theorem}[section]

\newtheorem*{p2.1}{Proposition 2.1}

\newtheorem{lemma}[theorem]{Lemma}

\theoremstyle{definition}
\newtheorem*{remark}{Remark}
\newtheorem*{remarks}{Remarks}
\newtheorem*{definition}{Definition}

\newcommand{\abs}[1]{\lvert#1\rvert}
\newcommand{\norm}[1]{\lVert#1\rVert}

\begin{document}

\title[Chebyshev Polynomials, I]{Asymptotics of Chebyshev Polynomials,\\I. Subsets of $\bbR$}
\author[J.~S.~Christiansen, B.~Simon, and M.~Zinchenko]{Jacob S.~Christiansen$^{1}$, Barry Simon$^{2,4}$,\\and
Maxim Zinchenko$^{3,5}$}

\thanks{$^1$ Centre for Mathematical Sciences, Lund University, Box 118, SE-22100, Lund, Sweden.
 E-mail: stordal@maths.lth.se}

\thanks{$^2$ Departments of Mathematics and Physics, Mathematics 253-37, California Institute of Technology, Pasadena, CA 91125.
E-mail: bsimon@caltech.edu}

\thanks{$^3$ Department of Mathematics and Statistics, University of New Mexico,
Albuquerque, NM 87131. E-mail: maxim@math.unm.edu}

\thanks{$^4$ Research supported in part by NSF grant DMS-1265592 and in part by Israeli BSF Grant No. 2010348}

\thanks{$^5$ Research supported in part by Simons Foundation grant CGM-281971.}

\

\date{\today}

\begin{abstract} We consider Chebyshev polynomials, $T_n(z)$, for infinite, compact sets $\fre \subset \bbR$ (that is, the monic polynomials minimizing the $\sup$--norm, $\norm{T_n}_\fre$, on $\fre$).  We resolve a $45+$ year old conjecture of Widom that for finite gap subsets of $\bbR$, his conjectured asymptotics (which we call Szeg\H{o}--Widom asymptotics) holds.  We also prove the first upper bounds of the form $\norm{T_n}_\fre \le Q C(\fre)^n$ (where $C(\fre)$ is the logarithmic capacity of $\fre$) for a class of $\fre$'s with an infinite number of components, explicitly for those $\fre \subset \bbR$ that obey a Parreau--Widom condition.
\end{abstract}

\maketitle

\section{Introduction} \lb{s1}

This paper is the first of what we hope will be a series studying the asymptotics of Chebyshev polynomials associated to an arbitrary compact subset, $\fre \subset \bbC$, which has an infinite number of points. These are those degree $n$ monic polynomials, $T_n$, which minimize
\begin{equation} \lb{1.1}
\norm{f}_\fre = \sup_{z \in \fre} |f(z)|
\end{equation}

See \cite{SY93} for background on general Chebyshev polynomials and their applications. It is known (see below for the case $\fre \subset \bbR$) that the minimizer is unique.  We will denote this minimizer as $T_n$ in cases where the intended $\fre$ is clear.  If $\mathrm{Re}P$ is the polynomial whose coefficients are the real parts of those of $P$, we have that $|\mathrm{Re}P(x)| < |P(x)|$ for all but finitely many $x \in \bbR$ unless $\mathrm{Re}P \equiv P$.  Since $\mathrm{Re}P$ is monic if $P$ is, $\mathrm{Re}T_n$ is also a Chebyshev polynomial so by the alternation theorem (see Theorem \ref{t1.1} below), $\vert\mathrm{Re}T_n(x)\vert=\Vert T_n\Vert$ has at least $n+1$ solutions in $\bbR$ if $\fre\subset\bbR$. Hence $\mathrm{Im}T_n\equiv 0$, that is, $T_n$ is real.  We begin by recalling a basic result which goes back to Borel \cite{Borel} and Markov \cite{Markov} (which according to Akhiezer \cite{Akh} was based on lectures from 1905 but only published in 1948).  It depends on a basic notion that comes from ideas of Chebyshev \cite{Cheb}:

\begin{definition} We say that $P_n$, a real degree $n$ polynomial, has an \textit{alternating set} in $\fre \subset \bbR$ if there exists $\{x_j\}^n_{j=0} \subset \fre$ with $x_0 < x_1 < \ldots < x_n$ so that
\begin{equation} \lb{1.2}
P_n(x_j) = (-1)^{n-j} \norm{P_n}_\fre
\end{equation}
\end{definition}

\begin{theorem} [The Alternation Theorem] \lb{t1.1} Let $\fre \subset \bbR$ be compact.  The Chebyshev polynomial of degree $n$ for $\fre$ has an alternating set in $\fre$.  Conversely, any monic polynomial with an alternating set in $\fre$ is the Chebyshev polynomial for $\fre$.
\end{theorem}

\begin{proof} The proof is simple and not so available in our generality, so we include it -- it is essentially what Markov gives in \cite{Markov} for the case $\fre = [a,b]$.  If $T_n$ is the Chebyshev polynomial, let $y_0 < y_1 < \ldots < y_k$ be the set of all the points in $\fre$ where its takes the value $\pm \norm{T_n}_\fre$. If there are fewer than $n$ sign changes among these ordered points, then we can find a degree at most $n-1$ polynomial, $Q$, non-vanishing at each $y_j$ and with the same sign as $T_n$ at those points. For $\epsilon$ small and positive, $T_n - \epsilon Q$ will be a monic polynomial with smaller $\norm{\cdot}_\fre$. Thus there must be at least $n$ sign flips and therefore an alternating set.

Conversely, let $P_n$ be a degree $n$ monic polynomial with an alternating set and suppose that $\norm{T_n}_\fre < \norm{P_n}_\fre$. Then at each point, $x_j$, in the alternating set for $P_n$, $Q \equiv P_n - T_n$ has the same sign as $P_n$,  so $Q$ has at least $n$ zeros, which is impossible, since it is of degree at most $n-1$.  \end{proof}

The alternation theorem implies uniqueness of the Chebyshev polynomial. For, if $T_n$ and $S_n$ are two minimizers, so is $Q \equiv \tfrac{1}{2}(T_n + S_n)$. At the alternating points for $Q$, we must have $T_n = S_n$, so they must be equal polynomials since there are $n+1$ points in the alternating set and their difference has degree at most $n-1$.

The alternation theorem also implies some simple facts about the zeros of $T_n$:

\begin{SL}
\item[{\rm{(a)}}] All the zeros of the Chebyshev polynomials of a set $\fre \subset \bbR$ lie in $\bbR$ and all are simple and lie in $\mathrm{cvh}(\fre)$, the convex hull of $\fre$.  This is because there must be at least one zero between any pair of points in an alternating set and this accounts for all $n$ zeros.  The same argument shows that for any $\gamma \in (-\norm{T_n}_\fre,\norm{T_n}_\fre)$ all $n$ solutions of $T_n(x) = \gamma$ are simple and lie in $\mathrm{cvh}(\fre)$.  This plus the open mapping theorem implies that (inverse as a function from $\bbC$)
\begin{equation} \lb{1.3}
\fre_n \equiv T_n^{-1} ([-\norm{T_n}_\fre,\norm{T_n}_\fre]) \subset \mathrm{cvh}(\fre)
\end{equation}
\item[{\rm{(b)}}] By a \textit{gap} of $\fre \subset \bbR$, we mean a bounded connected component of $\bbR \setminus \fre$. If there are only finitely many gaps and no component of $\fre$ is a single point, we speak of a \textit{finite gap set}. Between any two zeros of $T_n$, there is a point in the alternating set so each gap of $\fre \subset \bbR$ has at most one zero of $T_n$.
\item[{\rm{(c)}}] Above the top zero (resp. below the bottom zero) of $T_n$, $|T_n(x)|$ is monotone increasing.  It follows that $x_n = \sup_{y \in \fre} y$ (resp $x_0 = \inf_{y \in \fre} y$) so at the endpoints of $\mathrm{cvh}(\fre) \subset \bbR$ we have that $|T_n(x)| = \norm{T_n}_\fre$.
\smallskip
\end{SL}

To get ahead of our story, a key understanding in our analysis in this paper is that $\fre_n$ defined in \eqref{1.3} is the spectrum of a periodic Schr\"{o}dinger operator and up to normalization, $T_n$ is its spectral theory discriminant; see Section~\ref{s2}.

Going back at least to Szeg\H{o} \cite{Sz1924} is the idea that potential theory is essential to the study of Chebyshev polynomials.  To settle the notation we use, we recall some of the basic definitions.  References for the potential theory that we need include \cite{Helms,Land,Ran,EqMC,CompHA,Tsu}. Given a probability measure, $d\mu$, of compact support on $\bbC$, we define its \textit{Coulomb energy}, $\calE(\mu)$ by
\begin{equation} \lb{1.4}
\calE(\mu) = \int d\mu(x)\, d\mu(y)\, \log\, |x-y|^{-1}
\end{equation}
and we define the \textit{Robin constant} of a compact set $\fre \subset \bbC$ by
\begin{equation} \lb{1.5}
R(\fre) = \inf \{\calE(\mu) \mid \supp(\mu) \subset \fre \, \mathrm{and} \, \mu(\fre) = 1\}
\end{equation}

If $R(\fre)=\infty$, we say $\fre$ is a \textit{polar set} or has \textit{capacity zero}. If something holds except for a polar set, we say it holds q.e. (for \textit{quasi-everywhere}). The \textit{capacity}, $C(\fre)$, of $\fre$ is defined by
\begin{equation} \lb{1.6}
C(\fre) = \exp(-R(\fre)), \qquad  R(\fre) = \log (1/C(\fre))
\end{equation}

If $\fre$ is not a polar set, it follows from weak lower semicontinuity of $\calE(\cdot)$ and weak compactness of the family of probability measures that there is a probability measure whose Coulomb energy is $R(\fre)$.  Since $\calE(\cdot)$ is strictly convex on the probability measures, this minimizer is unique.  It is called the \textit{equilibrium measure} or \textit{harmonic measure} of $\fre$ and denoted $d\rho_{\fre}$.  The second name comes from the fact (see Conway \cite{Conway} or Simon \cite{CompHA}) that if $f$ is a continuous function on $\fre$, there is a unique function, $u_f$, harmonic on $(\bbC \cup \{\infty \})\setminus \fre$, which approaches $f(x)$ for q.e. $x \in \fre$ (i.e., solves the Dirichlet problem) and
\begin{equation} \lb{1.7}
u_f(\infty) = \int_{\fre}f(x)d\rho_{\fre}(x)
\end{equation}

The function $\Phi_{\fre}(z) = \int_{\fre}d\rho_{\fre}(x)\, \log\, |x-z|^{-1}$ is called the \textit{equilibrium potential}. The \textit{Green's function}, $G_{\fre}(z)$, of a compact subset, $\fre \subset \bbC$, is defined by
\begin{equation} \lb{1.8}
G_{\fre}(z) = R(\fre) - \Phi_{\fre}(z)
\end{equation}
It follows from Frostman's theorem that it is the unique function harmonic on $\bbC \setminus \fre$ with q.e. boundary value $0$ on $\fre$ and so that $G_{\fre}(z) - \log \, |z|$ is harmonic at $\infty$.  Moreover, $G_{\fre}(z) \ge 0$ everywhere and near $\infty$
\begin{equation} \lb{1.9}
G_{\fre}(z) = \log\, \abs{z} + R(\fre) + \mathrm{O}(1/\abs{z})
\end{equation}
equivalently,
\begin{equation} \lb{1.10}
\exp(G_\fre(z)) = \frac{\abs{z}}{C(\fre)} + \mathrm{O}(1)
\end{equation}
If $G_\fre$ is zero on $\fre$ and continuous on all of $\bbC$, we say that $\fre$ is \textit{regular (for potential theory)}.

To put our new results in context, we need to remind the reader of some previous results.  Using what is now called the Bernstein--Walsh lemma, Szeg\H{o} \cite{Sz1924} proved for all non-polar compact sets $\fre \subset \bbC$
\begin{equation} \lb{1.11}
\norm{T_n}_\fre \ge C(\fre)^n
\end{equation}
which was improved when $\fre \subset \bbR$ by Schiefermayr \cite{Schie} (see Section~\ref{s2}) to
\begin{equation} \lb{1.12}
\norm{T_n}_\fre \ge 2 C(\fre)^n
\end{equation}

Szeg\H{o} \cite{Sz1924}, using in part prior results of Faber \cite{Faber} and Fekete \cite{Fekete}, proved

\begin{theorem} [FFS Theorem] \lb{t1.2} For any compact non-polar set $\fre \subset \bbC$, one has that
\begin{equation} \lb{1.13}
\lim_{n \to \infty} \norm{T_n}_\fre^{1/n} = C(\fre)
\end{equation}

\end{theorem}

Upper bounds on $\norm{T_n}_\fre$ which complement \eqref{1.11} or \eqref{1.12} in that they also grow like $C(\fre)^n$ are clearly interesting.  The following is known

\begin{theorem} [Totik--Widom Theorem] \lb{t1.3} For any finite gap set $\fre \subset \bbR$, one has, for a constant $Q \ge 2$ depending on $\fre$, that
\begin{equation} \lb{1.14}
\norm{T_n}_\fre \le Q C(\fre)^n
\end{equation}
\end{theorem}

\begin{remarks} 1. This result follows from work of Widom \cite{Widom} that we discuss below on asymptotics of $\norm{T_n}_\fre$ (see Theorem~\ref{t1.8}).  Totik \cite{Totik09} proved an equivalent result but in a different form involving control of $C(\fre_n)$ where $\fre_n$ is the set of \eqref{1.3}.  We'll discuss this further in Section~\ref{s4}.

\smallskip

2. Neither approach leads to very explicit control on the constant $Q$ (although Widom explicitly finds $\limsup_{n \to \infty} \norm{T_n}_\fre/C(\fre)^n$ in terms of the solution to a minimization problem and he does have explicit bounds on this $\limsup$ but not on the $\sup$).

\smallskip

3. Widom proved this bound also for certain sets $\fre \subset \bbC$ that have finitely many components.  Recently, Andrievskii \cite{Andr} and Totik--Varga \cite{TV} have increased the family of finite component sets in $\bbC$ for which \eqref{1.14} holds.
\end{remarks}

One of our two main results in this paper extends this last result to a larger class of sets $\fre \subset \bbR$ with a simple explicit bound on $Q$ in terms of $G_\fre$.  Recall \cite{Parreau, Widom71}

\begin{definition} A set $\fre \subset \bbC$ is said to be a \textit{Parreau--Widom set} if
\begin{equation} \lb{1.15}
PW(\fre) \equiv \sum_{w \in \calC} G_\fre(w) < \infty
\end{equation}
where $\calC$ is the set of critical points of $G_\fre$ (i.e., points where $\partial G_\fre(w) = 0$)
\end{definition}

In this paper, we use Wittinger calculus:
\begin{equation} \lb{1.15a}
\partial = \frac{1}{2} \left( \frac{\partial}{\partial x} - i\, \frac{\partial}{\partial y}\right), \qquad
\bar\partial = \frac{1}{2} \left( \frac{\partial}{\partial x} + i\, \frac{\partial}{\partial y}\right)
\end{equation}
We recall that $\bar\partial f = 0$ are the Cauchy--Riemann equations (so that, for harmonic functions, $u$, $\partial u$ is analytic) and that for analytic functions, $f$, we have that $\partial f = f'$, the complex derivative.  Moreover, by the Cauchy--Riemann equations
\begin{equation} \lb{1.15b}
f \,\mathrm{analytic} \Rightarrow 2 \, \partial(\mathrm{Re}(f)) = f'
\end{equation}

If $\fre \subset \bbR$, it is easy to see all the critical points lie in $\bbR$.  If $\fre$ is also regular, there is exactly one critical point in each gap and so \eqref{1.15} is the sum over the maxima of $G_\fre$ in the gaps.  In particular, every finite gap set is a Parreau--Widom set.  Our new result, proven in Section~\ref{s4}, is:

\begin{theorem} \lb{t1.4} If $\fre \subset \bbR$ is a regular Parreau--Widom set, then
\begin{equation} \lb{1.16}
\norm{T_n}_\fre \le 2 \exp(PW(\fre)) C(\fre)^n
\end{equation}
\end{theorem}

\begin{remarks} 1.  For a finite gap set, the sum in \eqref{1.15} has finitely many terms, so is finite, and thus this result implies Theorem~\ref{t1.4} with a fairly explicit $Q$.  We note that homogeneous sets in the sense of Carleson \cite{Carl}, and, in particular, positive measure Cantor sets, are regular Parreau--Widom sets \cite{JM}.

\smallskip

2.  It is interesting to know for what other sets in $\bbR$ a bound like \eqref{1.14} is true.  For example, does the classical ${1}/{3}$ Cantor set, which is not a Parreau--Widom set, obey \eqref{1.14}?

\smallskip

3. We wonder if this result extends to Parreau--Widom sets in $\bbC$.

\smallskip

4. For the finite gap case, Widom \cite{Widom} obtains a bound on the $\limsup$ involving $\exp(PW(\fre))$ and our result is compatible with his in this finite gap case.

\end{remarks}

Our main focus will be on pointwise asymptotics of $T_n(z)$ on $\bbC \setminus \fre$.  The earliest results on this subject go back to Faber \cite{Faber} in 1919.  Let $\fre$ be a Jordan region with analytic boundary, i.e., an analytic Jordan curve together with its interior region. By the maximum principle, the Chebyshev polynomials for $\fre$ are the same as those for the curve. There is a unique Riemann map, $B_\fre(z)$, from $(\bbC \cup \{\infty \})\setminus \fre$ onto $\bbD$ which is a bijection with $B_\fre(\infty) = 0$ and positive ``derivative'', $B_\fre'(\infty)$, at $\infty$. Then:

\begin{theorem} [Faber \cite{Faber}] \lb{t1.5} If $\fre$ is a Jordan region with an analytic boundary, then
\begin{equation} \lb{1.17}
\lim_{n \to \infty} T_n(z) B_\fre(z)^n B_\fre'(\infty)^{-n} = 1
\end{equation}
uniformly for $z$ in a neighborhood of the closure of $(\bbC \cup \{\infty \})\setminus \fre$.
\end{theorem}

\begin{remarks} 1. Since the curve is assumed analytic, $B_\fre(z)$ has a continuation into a neighborhood of the curve.

\smallskip

2. Since $B_\fre$ maps the curve to $\partial \bbD$, and $G_\fre$ is unique, on $(\bbC \cup \{\infty \})\setminus \fre$, we have that
\begin{equation} \lb{1.18}
|B_\fre(z)| = \exp(-G_\fre(z))
\end{equation}
so that near $\infty$, we have that
\begin{equation} \lb{1.19}
B_\fre(z) = C(\fre) z^{-1} + \mathrm{O}(|z|^{-2})
\end{equation}
This implies that $B_\fre'(\infty) = C(\fre)$; thus, Faber's result implies that $\lim_{n \to \infty} \norm{T_n} C(\fre)^{-n} = 1$, a strong version of the FFS theorem.

\smallskip

3.  We call \eqref{1.17} \textit{Szeg\H{o} asymptotics} after Szeg\H{o}'s famous result \cite{Sz1920} on the aysymptotics of OPUC.  Since Faber's paper was earlier than Szeg\H{o}'s, this naming is perhaps unfair, but the term Szeg\H{o} asymptotics is so common, we use it in this case also.
\end{remarks}
In 1969, Widom wrote a 100+ page brilliant, seminal work \cite{Widom} on the asymptotics of Chebyshev and orthogonal polynomials associated to a set $\fre$, where $\fre$ is the union of a finite number of Jordan regions with $C^{2+}$ boundary and $C^{2+}$ Jordan arcs (i.e., not closed, simple curves).  As in Faber's case, the polynomials are the same whether one takes Jordan regions or their boundaries, as Widom does.

Widom began by looking for the replacement for $B_\fre(z)$ in \eqref{1.17}/\eqref{1.18}.  Since $G_\fre(z)$ is harmonic on $\bbC \setminus \fre$, it has a local harmonic conjugate so one can locally define an analytic function, $B_\fre(z)$, on $(\bbC \cup \{\infty \})\setminus \fre$ obeying \eqref{1.18} ($\infty$ is a removable singularity if one sets $B_\fre(\infty) = 0$.) $B_\fre(z)$ is determined by \eqref{1.18}, up to a phase which we can fix by demanding \eqref{1.19} near $\infty$.

$B_\fre(z)$ can be continued along any curve lying in $(\bbC \cup \{\infty \})\setminus \fre$ and, by the monodromy theorem, the continuation is the same for homotopic curves.  Since $G_\fre$ is continuous, only the phase can change, i.e., the phase change is associated with a character, $\chi_\fre$, of the fundamental group of $(\bbC \cup \{\infty \})\setminus \fre$.  The character is non-trivial if $\fre$ is not connected (up to polar sets) -- indeed, if a curve loops once around a subset $\frg$ of $\fre$, the phase change in $B_\fre$ is $\exp(-2\pi i \rho_\fre(\frg))$; see Theorem \ref{t2.7}.

There is a language introduced by Sodin--Yuditskii \cite{SY} for doing the bookkeeping for such functions.  It relies on the fact that the universal cover of $(\bbC \cup \{\infty \})\setminus \fre$ is $\bbD$.  Using the notation from our presentation of this machinery \cite{CSZ1}, there is a Fuchsian group, $\Gamma$, of M\"{o}bius transformations on $\bbD$, and a map $\x(z)$ from $\bbD \to (\bbC \cup \{\infty \})\setminus \fre$ which is automorphic (i.e., invariant under $\Gamma$).

$\x$ is a covering map, so a local bijection.  Its ``inverse'', $\z(x)$, is a multivalued analytic function which is not character automorphic -- rather its values are an orbit of the group $\Gamma$. $\Gamma$ is such that $\sum_{\gamma \in \Gamma} (1-\vert\gamma(0)\vert) < \infty$ so one can form the Blaschke product $B(z) = \prod_{\gamma \in \Gamma} b(z,\gamma(0))$.  In this language, one should use a complex variable, $x$, on $(\bbC \cup \{\infty \})\setminus \fre$ in which case one has that $B_\fre (x) = B(\z(x))$ and the object whose asymptotics we should look at is $T_n(\x(z)) B(z)^n C(\fre)^{-n}$ on $\bbD$.  While we feel this language should be in the back of one's mind, we will do our analysis with multivalued functions on $\Omega \equiv (\bbC \cup \{\infty \})\setminus \fre$, as Widom did.

So Widom looked at $T_n(z) B_\fre(z)^n C(\fre)^{-n}$.  Unlike the simply connected case of $\Omega$ studied by Faber, this cannot have a pointwise limit because the character of this character automorphic function is $\chi_\fre^n$ which is not constant!  Instead Widom found a good candidate for the asymptotics:


\enlargethispage*{12pt}

\begin{theorem} (Widom \cite{Widom}) \lb{t1.6} Let $\fre$ be a finite union of disjoint smooth Jordan regions and arcs. For every character, $\chi$, of the fundamental group of $\Omega$ there is a character automorphic function with that character, $F(z,\chi)$, on $\Omega$ which minimizes $\norm{f}_\Omega$ among all character automorphic functions, f, with that character and which obey $f(\infty,\chi) = 1$.  Moreover, this minimizer is unique and it and its $\norm{\cdot}_\Omega$ are continuous in $\chi$ (the functions in the topology of uniform convergence of compact subsets of the universal cover of $\Omega$).
\end{theorem}

\begin{remarks} 1.  We should refer to $F$ as a function on the universal cover of $\Omega$, not $\Omega$.

\smallskip

2. Continuity in $\chi$ and uniqueness are intimately related.

\smallskip

3. We will use $F_n(z)$ for the function $F(z,\chi_\fre^n)$.

\end{remarks}

The \textit{Widom surmise} is the notion that
\begin{equation} \lb{1.20}
\lim_{n \to \infty} \left[\frac{T_n(z) B_\fre(z)^n}{C(\fre)^n} - F_n(z)\right] = 0
\end{equation}
When it holds uniformly on compact subsets of the universal cover of $\Omega$, we will say that $\fre$ has \textit{Szeg\H{o}--Widom asymptotics}.

Widom proved two results about the asymptotics of $T_n$.  The first involves the situation where there are no arcs -- but only regions:

\begin{theorem} (Widom \cite{Widom}) \lb{t1.7}  Let $\fre$ be the union of a finite number of disjoint Jordan regions with smooth boundaries. Then $\fre$ has Szeg\H{o}--Widom asymptotics, i.e., \eqref{1.20} holds uniformly on compact subsets of the universal cover of $\Omega$.  Moreover,
\begin{equation} \lb{1.21}
\lim_{n \to \infty} \frac{\norm{T_n}_\fre}{C(\fre)^n \norm{F_n}_\Omega} = 1
\end{equation}
\end{theorem}

The second concerns finite gap sets in $\bbR$:

\begin{theorem} (Widom \cite{Widom}) \lb{t1.8} Let $\fre$ be a finite gap subset of $\bbR$.  Then
\begin{equation} \lb{1.22}
\lim_{n \to \infty} \frac{\norm{T_n}_\fre}{C(\fre)^n \norm{F_n}_\Omega} = 2
\end{equation}
\end{theorem}

Widom also conjectured that one had Szeg\H{o}--Widom asymptotics in this case.  At first sight this seems surprising -- \eqref{1.20} suggests that one might expect the limit in \eqref{1.22} to be 1 as it is in \eqref{1.21}.  Widom was clearly motivated by the example of $\fre = [-1,1]$ where both \eqref{1.20} and \eqref{1.22} hold!  Indeed, in that case $\x(z)$ is one half the Joukowski map, $\x(z) = \tfrac{1}{2}(z + z^{-1})$, and $B_\fre(x) = \z(x)$, the inverse of the one half the Joukowski map, i.e., $B_\fre(z) = z - \sqrt{z^2 - 1}$ and $1/B_\fre(z) = z + \sqrt{z^2 - 1}$.  The familiar formula for the Chebyshev polynomials in this case (a multiple the usual Chebyshev polynomials of the first kind with the multiple chosen to make the polynomials monic) is:
\begin{equation} \lb{1.23}
T_n(\cos(\theta)) = 2^{-n+1} \cos(n \theta), \qquad T_n(z) = 2^{-n} [B_\fre^n(z)+B_\fre^{-n}(z)]
\end{equation}
This implies that $\norm{T_n}_\fre = 2^{-n+1} = 2 C(\fre)^n$ since $C([-1,1]) = {1}/{2}$. This is consistent with the FSS Theorem and saturates Schiefermayr's bound \eqref{1.12}.  Since \eqref{1.18} holds and $G_\fre$ is $0$ (resp. $>0$) on $\fre$ (resp. off $\fre$), we have that $|B_\fre|$ is $1$ (resp. $<1$) on $\fre$ (resp. off $\fre$).  Thus off $\fre$, only $B_\fre^{-n}$ contributes to the asymptotics while on $\fre$, there are points with $B_\fre(z) = 1$ so both terms contribute and the norm is twice as large as one might have expected.  This explains where Widom's conjecture came from.  Our second main result here is a proof of this conjecture:

\enlargethispage*{12pt}
\begin{theorem} \lb{t1.9}  The Chebyshev polynomials for any finite gap set in $\bbR$ have Szeg\H{o}--Widom asymptotics. \end{theorem}

The hard work for this result was already done by Widom in proving Theorems~\ref{t1.6} and \ref{t1.8} and in our apriori result  Theorem~\ref{t1.4}.  It seems to us reasonable to think that any Parreau--Widom subset (perhaps with the additional requirement that the direct Cauchy theorem holds -- see the discussion in section 3 of Christiansen \cite {Chr}), of $\bbR$ has Szeg\H{o}--Widom asymptotics. The main issue is extending Theorems~\ref{t1.6} and \ref{t1.8} to that case.  Besides these results, we exploit the connection of Chebyshev polynomials to the spectral theory of periodic Jacobi matrices which we present in Section~\ref{s2}.  In Section~\ref{s3} we discuss several results about root asymptotics and we prove Theorem \ref{t1.4} in Section~\ref{s4} and Theorem~\ref{t1.9} in Section~\ref{s5}.

As an aside we note that the limit 2 in \eqref{1.22} is special to the case of $\fre \subset \bbR$ even though Widom had conjectured the limit was 2 as long as there was at least one arc (and not just regions) included among the Jordan arcs and regions.  Indeed, for the case where $\fre$ is a connected subset of the unit circle, the limit has been computed by Thiran--Detaille \cite{TD} who find it is always strictly between 1 and 2 if the connected set is a proper, non-empty subset.  Moreover, Totik--Yuditskii \cite{TY} have shown the $\limsup$ is strictly less than 2 if at least one Jordan region is included among the components of a set $\fre$ of Widom's class and Totik \cite{Totik14} has shown the $\liminf$ is strictly bigger than 1 if at least one Jordan arc is included among the components of a set $\fre$ of Widom's class.  But it still seems to us a reasonable, albeit difficult, conjecture that every set of Widom's class has Szeg\H{o}--Widom asymptotics.

We would like to thank V. Totik and P. Yuditskii for useful communications.  J.S.C. and M.Z. would like to thank D. Ramakrishnan and T. Soifer for the hospitality of Caltech where much of this work was done.


\vspace*{-8pt}
\section{Periodic Sets} \lb{s2}

In this section, we'll see the important role played by the sets, $\fre_n$, of \eqref{1.3} and the related sets:
\begin{equation} \lb{2.1}
\stackrel{\circ}{\fre}_n \equiv T_n^{-1} ((-\norm{T_n}_\fre,\norm{T_n}_\fre))
\end{equation}
Clearly, by the definition of $T_n$ and $\fre_n$, we have that
\begin{equation} \lb{2.2}
\fre \subset \fre_n
\end{equation}

We will see that the set $\fre_n$ determines many properties of $T_n$.  In particular,
\begin{equation} \lb{2.3}
\norm{T_n}_\fre = 2 C(\fre_n)^n
\end{equation}
which, by \eqref{2.2}, implies Schiefermayr's bound, \eqref{1.12}.  If $B_n$ is short for $B_{\fre_n}$, we'll also prove (indeed, we'll use this to prove \eqref{2.3})
\begin{equation} \lb{2.4}
\frac{2 T_n (z)}{\norm{T_n}_\fre} = B_n(z)^n + B_n(z)^{-n}
\end{equation}
Given our discussion of Szeg\H{o}--Widom asympotics for $\fre = [-1,1]$, it should not be a surprise that \eqref{2.4} is a significant part of our proof of Theorem \ref{t1.9}.

The equilibrium measure for $\fre_n$ which we'll denote $\rho_n$ will also play a role.  We'll prove that, for any gap $K$ of $\fre$, one has that
\begin{equation} \lb{2.5}
\rho_n(K) \le 1/n
\end{equation}
which will be the key to our proof of Theorem \ref{t1.4}.

An interesting further fact concerns the weight that $\rho_n$ gives to components of $\fre_n$.  We will call a compact set $\frg \subset \bbR$ a \textit{period-$n$ set} if and only if each connected component of $\frg$ has harmonic measure $k/n$ for some $k \in \{1,\ldots,n\}$ (which, of course implies that $\frg$ has at most $n$ components and so is a finite gap set).  We will prove that any $\fre_n$ is a period-$n$ set and that conversely, if $\fre$ is a period-$n$ set, then it is its own $\fre_n$.

The name ``period-$n$ set'' comes from the fact that these sets are precisely the spectra of two-sided periodic Jacobi matrices. The original proofs we had for some of the results we just described used the theory of such matrices and we have kept some of the terminology.  While we will prove these results here using only the alternation theorem and some potential theory, we'll end the section with a brief indication of the approach that relies on the fact that ${2 T_n}/{\norm{T_n}_\fre}$ is the discriminant of a periodic Jacobi matrix.

We are not the first ones to note the special properties of polynomials, $P$, for which $P^{-1}([-A,A]) \subset \bbR$.  Their use is implicit in much of the work on the theory of periodic Schr\"{o}dinger operators and Jacobi matrices as we'll explain at the end of this section. In the orthogonal polynomial community, there is an initial work of Geronimo--Van Assche \cite{GvA} and important follow-up of Peherstorfer \cite{Per0, Per1, Per2, Per3} and Totik \cite{Totik00, Totik01, Totik09, Totik12}.

\begin{theorem} \lb{t2.1} Let $\fre$ be an infinite, compact subset of $\bbR$, $T_n$ its $n$th Chebyshev polynomial and let $\fre_n$ and $\stackrel{\circ}{\fre}_n$ be given respectively by \eqref{1.3} and \eqref{2.1}.  Then there exist $\alpha_1 < \beta_1 \le \alpha_2 < \dots \beta_j \le \alpha_{j+1} \dots < \beta_n$ so that
\begin{equation} \lb{2.6}
\stackrel{\circ}{\fre}_n = \bigcup_{j=1}^n (\alpha_j, \beta_j), \qquad \fre_n = \bigcup_{j=1}^n [\alpha_j, \beta_j]
\end{equation}
Moreover on $(\alpha_j, \beta_j)$, we have that $(-1)^{n-j} T_n'(x) > 0$, $\{\alpha_1,\beta_n\} \in \fre$ and for each $j = 1,\dots,n-1$, at least one of $\beta_j$ and $\alpha_{j+1}$ lie in $\fre$.
\end{theorem}

\begin{proof} As we noted in the consequences of the alternation theorem, for any $\gamma \in (-\norm{T_n}_\fre,\norm{T_n}_\fre)$ all $n$ solutions of $T_n(x) = \gamma$ are simple and lie in $\mathrm{cvh}(\fre)$.  This implies the claimed structure for $\stackrel{\circ}{\fre}_n$ and $\fre_n$, \eqref{2.6} and the derivative condition.  The $\alpha$'s and $\beta$'s are all the solution of $T_n(x) = \pm \norm{T_n}_\fre$ so the remainder of the theorem is a restatement of the alternation theorem. \end{proof}

We will call $[\alpha_j,\beta_j] = \fre_n^{(j)}$, the \textit{$j$th band} of $\fre_n$.  Define
\begin{equation} \lb{2.7}
\Delta_n(z) \equiv \frac{2 T_n (z)}{\norm{T_n}_\fre}
\end{equation}
so that $\fre_n$ is exactly the set where $-2 \le \Delta_n(x) \le 2$ and $\Delta_n$ takes values in $\bbC \setminus [-2,2]$ on $\bbC \setminus \fre_n$.  The Joukowski map $z \mapsto z+z^{-1}$ takes $\bbD$ one-one to $\bbC \setminus [-2,2]$ and $\partial\bbD$ two-one to $[-2,2]$ so its functional inverse $z \mapsto \tfrac{z}{2} - \sqrt{\left(\tfrac{z}{2}\right)^2 - 1}$ maps $(\bbC \cup \{\infty \}) \setminus [-2,2]$ to $\bbD$.  The numerical inverse of this, $z \mapsto \tfrac{z}{2} + \sqrt{\left(\tfrac{z}{2}\right)^2 - 1}$, thus maps $(\bbC \cup \{\infty \}) \setminus [-2,2]$ to $(\bbC \cup \{\infty \})\setminus \overline{\bbD}$.  It follows that
\begin{equation}
\frac{\Delta_n(z)}{2} + \sqrt{\left(\frac{\Delta_n(z)}{2}\right)^2 - 1}
\end{equation}
maps $\Omega_n \equiv (\bbC \cup \{\infty \})\setminus \fre_n$ to $(\bbC \cup \{\infty \})\setminus \overline{\bbD}$.  If we take the log of the absolute value of this nonvanishing analytic function, we get a strictly positive harmonic function on $\bbC \setminus \fre_n$.  Since this function approaches 0 as one approaches $\fre_n$ and is $n \log|z| + \mathrm{O}(1)$ near $\infty$, we have proven the first assertion in:

\begin{theorem} \lb{t2.2} Let $\fre$ be an infinite compact subset of $\bbR$, $T_n$ its $n$th Chebyshev polynomial, $\Delta_n$ given by \eqref{2.7} and let $\fre_n$ be given by \eqref{1.3}. Then the Green's function, $G_n$, of $\fre_n$ is given by:
\begin{equation} \lb{2.8}
G_n(z) = \frac{1}{n} \log \left|\frac{\Delta_n(z)}{2} + \sqrt{\left(\frac{\Delta_n(z)}{2}\right)^2 - 1}\right|
\end{equation}
Moreover, we have that:
\begin{equation} \lb{2.9}
\begin{split}
 B_n(z)^n = \frac{\Delta_n(z)}{2} - \sqrt{\left(\frac{\Delta_n(z)}{2}\right)^2 - 1} \\
 B_n(z)^{-n} = \frac{\Delta_n(z)}{2} + \sqrt{\left(\frac{\Delta_n(z)}{2}\right)^2 - 1}
 \end{split}
\end{equation}
and \eqref{2.4} and \eqref{2.3} hold.
\end{theorem}

\begin{proof} We proved \eqref{2.8} above.  By that formula, the absolute value of the right-hand side of the second equation in \eqref{2.9} is $\exp(nG_\fre(z))$.  Since this expression is analytic on $\Omega_n \setminus \{\infty\}$ and is $C z^n + \mathrm{O}(z^{n-1})$ with $C > 0$ there, it must be $B_n(z)^{-n}$.  The first equation in \eqref{2.9} holds since both sides are inverses of the two sides of the second equation, which we have just proven.  Adding the two equations in \eqref{2.9} and using \eqref{2.7}, we get \eqref{2.4}.

By \eqref{1.10}, $B_n(z)^{-n} = z^n C(\fre_n)^{-n} +  \mathrm{O}(z^{n-1})$; we see that $\Delta_n(z) = z^n C(\fre_n)^{-n} +  \mathrm{O}(z^{n-1})$ also.  By \eqref{2.7} and the fact that $T_n$ is monic, we obtain \eqref{2.3}.  \end{proof}

Since $\fre \subset \fre_n$, we have that $C(\fre) \le C(\fre_n)$, so \eqref{2.3} immediately implies Schiefermayr's Theorem, \eqref{1.12}.

Next, we turn to the form of the equilibrium measure, $\rho_n$, for $\fre_n$.  We note that $\Delta_n$ runs monotonically from $-2$ to $+2$ or vice versa. We have that:

\begin{theorem} \lb{t2.3} In each band of $\fre_n$, define $\theta(x) \in [0, \pi]$ by
\begin{equation} \lb{2.10}
\Delta_n(x) =2 \cos(\theta(x))
\end{equation}
Then
\begin{equation} \lb{2.11}
d\rho_n (x) = (\pi n)^{-1} |\theta'(x)| dx
\end{equation}
In particular, each band has $\rho_n$-measure ${1}/{n}$.  If $\eta_j \in \fre_n^{(j)}$ is the zero of $T_n$ in $\fre_n^{(j)}$, then each of $[\alpha_j, \eta_j]$ and $[\eta_j, \beta_j]$ has $\rho_n$-measure ${1}/{2n}$.
\end{theorem}

We will not give a formal proof of this result. The final sentence is an immediate consequence of \eqref{2.11} given that $\Delta_n$ runs monotonically from $2$ to $-2$ or from $-2$ to $2$ on a band, so that $\theta$ runs monotonically from $0$ to $\pi$ or from $\pi$ to $0$.  \eqref{2.11} is well known in the mathematical physics literature obtained from the theory of discriminants.  For example, Simon \cite{SimonSzego} has two proofs of it -- one as Theorem 5.3.8 and one as Theorem 5.4.8.  A quick proof is to apply the operator $\partial$ of \eqref{1.15a} to \eqref{2.8}, using \eqref{1.15b} to get
\begin{equation} \lb{2.12}
\int \frac{d\rho_n(x)}{x-z} = \frac{1}{n} \frac{\Delta_n'(z)}{\sqrt{\Delta_n(z)^2 - 4}}
\end{equation}
Taking imaginary parts of both sides, one gets \eqref{2.11} by noting the boundedness of this imaginary part and computing its boundary value.  The square root on the right of \eqref{2.12} is pure imaginary on $\fre_n$.  One needs to track carefully its phase from the square root singularity which is compensated in the ratio by the change of the sign of $\Delta_n'$ from band to band.  This immediately implies a strong form of \eqref{2.5}.

\begin{theorem} \lb{t2.4} Let $K$ be a gap of $\fre$.  Then \eqref{2.5} holds.  If $T_n$ has no zero in $K$, then $1/n$ can be replaced by $1/2n$.  Moreover, $K \cap \fre_n$, if non-empty, is a single interval.
\end{theorem}
\begin{remarks} 1. When $\fre$ is a finite gap set, it is an implicit result of Sodin-Yuditskii \cite {SY93} and explicit result of Peherstorfer \cite{Per0, Per2} that each gap contains no more than one band.

\smallskip

2. The interval mentioned in the last sentence may be closed (if the band is entirely in $K$), half open (if one end of the intersection is an end-point of $K$), or open (if the intersection is all of K).

\smallskip

3. From \eqref{2.12}, we deduce that $\rho_n$ is a.c. with respect to $dx$ and
\begin{equation} \lb{2.12a}
d\rho_n(x) = w_n(x) dx , \qquad w_n(x) = \frac{1}{\pi n}\frac{|\Delta_n'(x)|}{\sqrt{4-\Delta_n(x)^2}}
\end{equation}
for $x \in \fre_n$
\end{remarks}
\begin{proof} Suppose that $\fre_n^{(j)} \cap K \ne \emptyset$.  Since $K$ is connected and at least one of $\beta_j$ or $\alpha_{j+1}$ lies in $\fre$, we conclude that K is disjoint from all the $\fre_n^{(k)}$ for $k > j$.  Similarly, $K$ is disjoint from all the $\fre_n^{(k)}$ for $k < j$.  Thus $K$ contains at most one band, so \eqref{2.5} follows from Theorem \ref{t2.3}.  If $T_n$ has no zero in K, at most half a band lies in $K$ and we get the improved ${1}/{2n}$ result. Since $K \cap \fre_n$, if non-empty, is a single band, we get the single interval claim.
\end{proof}

 Theorem \ref{t2.3} has another immediate consequence:

\begin{theorem} \lb{t2.5} $\fre_n$ is a period-$n$ set.
\end{theorem}

Our penultimate result in this section is a converse to this result.  We need two preliminaries:

\begin{theorem} \lb{t2.6} Suppose that $[\alpha,\beta]$ is a connected component of a compact set $\fre \subset \bbR$.  Then:
\begin{SL}
\item[{\rm{(a)}}] $G_\fre$ has an analytic continuation across $(\alpha,\beta)$, i.e., there is an analytic function in a neighborhood, $N$, of $(\alpha,\beta)$ whose real part agrees with $G_\fre$ on $\{z \in N \,|\, \mathrm{Im}(z) > 0\}$.  $G_\fre$ vanishes  everywhere on $(\alpha,\beta)$ and is continuous on $\bbC_\pm\cup(\alpha,\beta)$.

\item[{\rm{(b)}}] If $\partial$  is given by \eqref{1.15a}, then $h(z) \equiv \sqrt{(z-\alpha)(\beta-z)} \partial G_\fre $ has an analytic continuation across $[\alpha,\beta]$, i.e., there is an analytic function in a neighborhood, $N_1$, of $[\alpha,\beta]$ which agrees with $h$ on $\{z \in N_1 \,|\, \mathrm{Im}(z) > 0\}$.

\item[{\rm{(c)}}] We have that
\begin{equation} \lb{2.13}
d\rho_\fre \restriction [\alpha,\beta] = \frac{q(x)}{\sqrt{(x-\alpha)(\beta-x)}} dx
\end{equation}
where $q(x) > 0$ on $(\alpha,\beta)$ and continuous on $[\alpha,\beta]$.
\item[{\rm{(d)}}] Suppose that $\fre = \bigcup_{k=1}^{p}[a_k,b_k]$.  Let $\{c_k\}_{k=1}^{p-1}$ be the critical points of $G_\fre$ where $a_k<b_k<c_k<a_{k+1}<b_{k+1}$, $k=1,\dots,p-1$.  Then $d\rho_\fre(x) = w(x) dx$ where, for $x \in \fre$,
\begin{equation} \lb{2.13a}
w(x) = \frac{\frac{1}{\pi}\prod_{k=1}^{p-1}|x-c_k|} {\prod_{k=1}^{p}|(x-a_n)(x-b_k)|^{1/2}}
\end{equation}
\end{SL}
\end{theorem}

\begin{remarks} 1. We note the compatibility of \eqref{2.12a} and \eqref{2.13a}.  For the leading coefficient of $\Delta_n'$ is $n$ times that of $\Delta_n$ canceling the ${1}/{n}$ yielding a formula like \eqref{2.13a} but with the product over all the zeros of $\Delta_n'$ in the numerator and over all band edges in the denominator.  At a closed gap, $4 -\Delta_n^2$ has a double zero and $\Delta_n'$ a single so they cancel and \eqref{2.13a} results in the special case where $\fre$ is a period-$n$ set.

\smallskip

2. $\rm{(d)}$ is, as we'll note, equivalent to a product formula for $2\partial G_\fre$.  This formula can be found, for example, as (5.4.88) in Simon \cite{SimonSzego}.

\smallskip

3. \eqref{2.15} below is not literally true but is a bit of poetry because $Q(x) = 1$ above $\rm{cvh}(\fre)$, so the integral in \eqref{2.15} diverges.  One can use the renormalized version of the Herglotz representation, only look at imaginary parts, or put in a cutoff.  For use in proving \rm{(b)}, the cutoff is no problem since the remainder is analytic in a neighborhood of $[\alpha,\beta]$.  For \rm{(d)}, if the upper cutoff is $R$ above $\rm{cvh}(\fre)$, we get a $\log(R-z)$ term which is $\log(R) + \mathrm{o}(1)$ so if we absorb $\log(R)$ into redefining $C$, we get a limit and the argument that we give then works.
\end{remarks}

\begin{proof}  These results are well known to experts on potential theory and/or spectral theory of Schr\"{o}dinger operators.  Especially relevant are ideas of Craig \cite{Craig} given that $G_\fre(x) = 0$ on $\fre$ implies that $\partial G_\fre$ is the Stieltjes transform of a measure reflectionless on $\fre$.  (We caution the reader that Craig's ``Green's function'' is not $G_\fre$ but $2 \partial G_\fre$.) So we'll only sketch the details.

Since $G_\fre$ is a positive harmonic function on the upper half plane, $\bbC_+$, there is a Herglotz function $f$ on $\bbC_+$ with $\mathrm{Im}(f) = G_\fre$, so we can write a Herglotz representation for it.  Since $G_\fre$ is locally bounded, the measure in this representation is absolutely continuous.  Moreover, since q.e. on $(\alpha,\beta)$, $\lim_{\epsilon \downarrow 0} G_\fre(x+i\epsilon) = 0$, this measure gives zero weight to $[\alpha,\beta]$ which implies $\rm{(a)}$.

By differentiating the formula for $G_\fre$ in terms the potential of $d\rho_\fre$, we see, by \eqref{1.15b}, that:
\begin{equation} \lb{2.14}
F(z) \equiv 2 \partial G_\fre(z) = \int \frac{d\rho_\fre(x)}{x-z}
\end{equation}
$\rm{(a)}$ implies that $F$ is analytic across $(\alpha,\beta)$ and ${\rm{Re}}(F(x)) = 0$ there.  Thus we can use ideas of Craig \cite{Craig} to write a Herglotz representation for $\log(F)$:
\begin{equation} \lb{2.15}
\log(F(z)) = C + \int \frac{Q(x) dx}{x-z}
\end{equation}
where $C$ is a real constant and $Q(x) = \tfrac{1}{\pi} \lim_{\epsilon \downarrow 0} \mathrm{Arg}(F(x+i \epsilon))$.  It follows that $Q(x) = {1}/{2}$ on $(\alpha,\beta)$ and is either identically 0 or identically 1 just below $\alpha$ and similarly just above $\beta$.  Exponentiating \eqref{2.15} implies $\rm{(b)}$ which easily leads to $\rm{(c)}$ with $q(x) \ge 0$.  To see that $q(x) > 0$, we note that the Herglotz representation \eqref{2.15} implies that ${\rm{Im}}(\log(F(z)) - \tfrac{\pi}{2}$ goes to zero as $(\alpha,\beta)$ is approached from the upper half plane.  By the strong reflection principle, $\log(F(z))$ has continuous boundary values, so in particular $F$ has no zeros on $(\alpha,\beta)$.  Since ${\rm{Re}}(F) = 0$ there, we see that ${\rm{Im}}(F)$ is non-vanishing there.

To get $\rm{(d)}$, we use the representation \eqref{2.15} (noting that $Q(x)$ is ${1}/{2}$ on each $(a_j,b_j)$, $-1$ on $(b_p, \infty)$ and on each $(b_j, c_j)$ and $1$ on $(-\infty, a_1)$ and each $(c_j,a_{j+1})$) to get a product formula for $F$.  This is equivalent to the formula for $w$ by the integral representation in \eqref{2.14} and the theory of boundary values of Stieltjes transforms.  (The constant $C$ in \eqref{2.15} is determined by the $-{1}/{z}$ asymptotics of $F$.)  \end{proof}

\begin{theorem} \lb{t2.7}
If $\fre \subset \bbC$ is compact and $\gamma$ is any rectifiable curve in $\bbC \setminus \fre$, then, $\Delta_\gamma(B_\fre)$, the change in phase of $B_\fre$ in going around $\gamma$, is given by
\begin{equation} \lb{2.16}
\Delta_\gamma(B_\fre) = \exp \left(-2\pi i \int N(\gamma, x) d\rho_\fre(x)\right)
\end{equation}
where $N(\gamma, x)$ is the winding number of $\gamma$ around x.  In particular, if $\gamma$ winds once around $\frg \subset \fre$ and around no other points of $\fre$, then the multiplicative change of phase of $B_\fre$ around $\gamma$ is $\exp(-2 \pi i \rho_\fre(\frg))$.
\end{theorem}

\begin{proof} Applying $2 \partial$ to both sides of $\log(|B_\fre|) = - G_\fre$, using the formula \eqref{1.8} for $G_\fre$ in terms of $\rho_\fre$ and \eqref{1.15b}, we get that
\begin{equation} \lb{2.17}
B_\fre'(z) B_\fre^{-1}(z) = \int \frac{d\rho_\fre(x)}{x-z}
\end{equation}
where one needs an easy argument to justify interchanging the derivative and integral.  Multiplying by $(2 \pi i)^{-1}$ and doing the contour integral, one gets \eqref{2.16} after interchanging the integrals and using the formula for $N(\gamma,x)$ as a contour integral.
\end{proof}

\begin{theorem} \lb{t2.8} Let $\fre \subset \bbR$ be a period-$n$ set.  Then for $k = 1,2,\dots$, $\fre$ is the set where its Chebyshev polynomial, $T_{kn}$, takes its values in $[-\norm{T_{kn}}_\fre, \norm{T_{kn}}_\fre]$, i.e.,  $\fre_{kn} = \fre$.
\end{theorem}
\begin{remark} It is easy to see that if $S_n$ is the Chebyshev polynomial for $[-1,1]$ (which is the classical Chebyshev polynomial of the first kind up to a constant), then for the $\fre$'s of this theorem, one has that $T_{kn} = \norm{T_n}_\fre^k S_k(T_n/\norm{T_n}_\fre)$.
\end{remark}
\begin{proof} By Theorem \ref{t2.7}, the argument of $B_\fre^n$ changes by an integral multiple of $2\pi$ as one goes around any connected component of $\fre$ so it defines a function analytic in $\bbC \setminus \fre$.  Since this function is real on $\bbR$ near $+\infty$, we have on $\bbC \setminus \fre$ that $B_\fre^n(\bar{z}) = \ol{B_\fre^n(z)}$.  Moreover, by Theorem \ref{t2.6}, this function is continuous as $\fre$ is approached from one or the other side of $\fre$ and has magnitude $1$ there.  This shows that
\begin{equation} \lb{2.18}
P_n (z) \equiv C(\fre)^n \left(B_\fre^n(z) + B_\fre^{-n}(z)\right)
\end{equation}
is continuous across the interior of $\fre$ and so analytic there.  The end points of the intervals are thus removable singularities since $P_n$ is bounded there by Theorem \ref{t2.6}.  It follows that $P_n$ is an entire function and, by the asymptotics, \eqref{1.19}, of $B_\fre$, it is a monic polynomial of degree n.

Since $|B_\fre| = 1$ on $\fre$, we have that $\norm{P_n}_\fre \le 2 C(\fre)^n$, so by Schiefermayr's inequality, \eqref{1.12}, and uniqueness of the minimizer, $P_n$ is $T_n$.  Since $|B_\fre| < 1$ on $\bbC \setminus \fre$, we see that $\fre_n = \fre$ as claimed.  This proves the $k=1$ part of the Theorem.  But any period-$n$ set is also a period-$kn$ set.  \end{proof}

Our final result in this section proves a minimality property of $\fre_n$.

\begin{theorem} \lb{t2.9} Let $\fre \subset \bbR$.  Then for any period-$n$ set, $\frg \supset \fre$, we have that
\begin{equation} \lb{2.19}
C(\fre_n) \le C(\frg)
\end{equation}
with equality if and only if $\frg = \fre_n$.
\end{theorem}

\begin{proof} Let $H_n$ be the $n$th Chebyshev polynomial for $\frg$.  Since $H_n$ is monic, we have that
\begin{equation} \lb{2.20}
2 C(\fre_n) = \norm{T_n}_\fre \le \norm{H_n}_\fre \le \norm{H_n}_\frg = 2 C(\frg)
\end{equation}
proving \eqref{2.19}.  If one has equality in \eqref{2.19}, then one has equality in the first inequality in \eqref{2.20}, so $T_n = H_n$ which implies, by Theorem \ref{t2.8} and the definition of $\fre_n$, that $\frg = \fre_n$.  \end{proof}

The results of this section can be understood from a spectral theory point of view. We end this section with a description of this connection to periodic Jacobi matrices -- one place to find the details of the theory of such matrices is Chapter~5 of Simon \cite{SimonSzego}. We consider two-sided sequences $\{a_j, b_j\}^\infty_{j=-\infty}$  with $a_j > 0, \, b_j \in \bbR$ and so that for some $p > 0$ and all $j$ in $\bbZ$
\begin{equation} \lb{2.21}
a_{j+p} = a_j, \qquad b_{j+p}=b_j
\end{equation}
We define doubly infinite tridiagonal matrices, $J$, with $b_j$ along the diagonal and $a_j$ on the principle subdiagonals (so that row $k$ has non-zero elements $a_{k-1} \; b_k \; a_k$ with $b_k$ in column $k$).

For $z \in \bbC$ fixed, we are interested in solutions, $\{u_j\}^\infty_{j=-\infty}$, of
\begin{equation} \lb{2.22}
a_j u_{j+1} + b_j u_j + a_{j-1} u_{j-1} = z u_j
\end{equation}
We study the $p$-step transfer (aka update) matrix:
\begin{equation} \lb{2.23}
M_p(z)\begin{pmatrix} u_1\\a_0u_0 \end{pmatrix} = \begin{pmatrix} u_{p+1}\\a_pu_p \end{pmatrix}
\end{equation}
We put $a$'s in the bottom component so that the one step matrix $ \frac{1}{a_j}\Bigl(\begin{smallmatrix} z-b_j & -1 \\a_j^2 & 0 \end{smallmatrix}\Bigr) $ has determinant $1$ and thus $\det(M_p(z)) = 1$.

In terms of the first and second kind orthogonal polynomials for Jacobi parameters $\{a_n, b_n\}^\infty_{n=1}$, as defined in Section 3.2 of \cite{SimonSzego},
\begin{equation} \lb{2.24}
M_p(z) = \begin{pmatrix} p_p(z) & -q_p(z)\\
a_pp_{p-1}(z) & - a_pq_{p-1}(z)\end{pmatrix}
\end{equation}

The \textit{discriminant}, $\Delta(z)$, defined by
\begin{equation} \lb{2.25}
\Delta(z) = \tr\bigl(M_p(z)\bigr) = p_p(z)-a_pq_{p-1}(z)
\end{equation}
is a (real) polynomial of degree exactly $p$. Given the recursion relations for $p_j(z)$ or the form of the one step transfer matrix, we see that $\Delta(z)$ is a polynomial of degree $p$ with leading coefficient $(a_1 \cdots a_p)^{-1}$.

If $M_p(z)$ has an eigenvalue $\lambda$, it is easy to see the difference equation has a (Floquet) solution obeying
$u_{j+mp} = \lambda^m u_j$ for all $m \in \bbZ$.  Since $\det(M_p(z)) = 1$, if $\lambda \ne \pm 1$, we get two linearly independent solutions, so if $|\lambda| \ne 1$, all solutions are exponentially growing at $\infty$ and/or at $-\infty$.  On the other hand, if $|\lambda| = 1$, there is a bounded solution.  Note that $M_p(z)$ has an eigenvalue with $|\lambda| = 1$ if and only if $\Delta(z) \in [-2,2]$.

Since it is known that the spectrum of $J$ is the closure of the set of $z$'s for which there are polynomially bounded solutions (Schnol's Theorem), we conclude that spec$(J)= \Delta^{-1}([-2,2])$. Since $J$ is self-adjoint, we have that $\Delta^{-1}([-2,2]) \subset \bbR$.

If $f(z)$ is an entire function real on the real axis and $f^\prime(x_0) = 0$ for $x_0 \in \bbR$, because of the local structure of analytic functions, there will be non-real $z$'s near $x_0$ with $f(z)$ a real value near $f(x_0)$.  Thus $\Delta^{-1}([-2,2]) \subset \bbR$ implies that
$$\Delta(x) \in (-2,2) \, \Rightarrow \, \Delta^\prime(x) \ne 0$$
Therefore, between successive points where $\Delta(x_0) = \pm 2$ and where $\Delta(x_1) = \mp 2$, $\Delta(x)$ is strictly monotone and $\Delta$ is a bijection.  It follows that $\Delta$ has an alternating set in $\fre =\Delta^{-1}([-2,2])$. Therefore, $a_1 \cdots a_p \Delta$ is the Chebyshev polynomial for $\fre = \fre_p$. It is known that $J$ has purely absolutely continuous spectrum of multiplicity 2, which implies that the half line operator is regular in the sense of Stahl--Totik \cite{ST} so one has that $(a_1 \cdots a_p)^{1/p} = C(\fre)$ and the density of zeros is $d\rho_\fre$. In the spectral theory literature, the density of zeros is called the density of states and it is well known that each ``band'' of the spectrum has density $1/n$.

From a spectral theory point of view, the fact that every period-$n$ set has a discriminant follows from the fact that such a set is the spectrum of a periodic Jacobi matrix.  Indeed, if $\ell$ is the number of gaps of a period-$n$ set, then one constructs an $\ell$ dimensional torus of such periodic matrices. For any finite gap set there is an isospectral torus which can be constructed as reflectionless Jacobi matrices (see Remling \cite{Rem}), or as minimal Herglotz functions (see our paper \cite{CSZ1}) or using Hardy spaces of character automorphic functions (see Sodin--Yuditskii \cite{SY}).  The elements of the isospectral torus are almost periodic with frequencies generated by the harmonic measures of the components of the finite gap set and so periodic with period $n$ if all these measures are of the form $k/n$.


\section{Root Asymptotics} \lb{s3}

In this section, before turning to our two main theorems in the final sections, we make a comment on the FSS theorem and a remark on a Theorem of Saff--Totik concerning root asymptotics of the Chebyshev polynomials of any arbitrary infinite, compact set $\fre \subset \bbC$.  The first concerns the following theorem of Totik:

\begin{theorem} \lb{t3.1} (Totik \cite{Totik00}) Given any infinite, compact subset $\fre \subset \bbR$, there exist period-$n$ sets $\frg_n \supset \fre$ so that
\begin{equation} \lb{3.1}
\lim_{n \to \infty} C(\frg_n) = C(\fre)
\end{equation}
\end{theorem}

\begin{remarks} 1. This is useful because one can use polynomial mappings to extend some results from $[-1,1]$ to period-$n$ sets and then this theorem to extend the result to general sets in $\bbR$.  Polynomial inequalities have been obtained by Totik using this method (see his review article \cite{Totik12}) and, using this method, Lubinsky's approach \cite{Lub} to universality for the CD kernel has been extended from $[-1,1]$ to general compact sets in $\bbR$ by Simon \cite{SimonLub} and Totik \cite{Totik08}.

\smallskip

2. A stronger result is known for finite gap sets -- namely, if $\fre$ has $\ell$ gaps, then for $n \ge \ell$, $\frg_n$ can be picked to also have exactly $\ell$ gaps.  Indeed, this is how Totik proved Theorem \ref{t3.1}.  This stronger result has been discovered and proven independently by several different authors \cite{Rob, McKvM, Bog, Per2, Totik00}.
\end{remarks}

Our point here is to note the following:

\medskip

\centerline{\textit{Theorem \ref{t3.1} is equivalent to the FSS Theorem}}

\medskip

Totik informed us that he knew this but it seems not to be in the literature.  To see that FSS $\Rightarrow$ Theorem \ref{t3.1}, note that one can take $\frg_n = \fre_n$ and use \eqref{2.3} and \eqref{1.13} plus $2^{1/n} \to 1$ to get \eqref{3.1}.  Conversely, given Theorem \ref{t3.1} and \eqref{2.19}, we see that \eqref{3.1} holds for $\frg_n = \fre_n$.  Then \eqref{2.3} and $2^{1/n} \to 1$ implies \eqref{1.13}.

\smallskip

Our other result on root asymptotics is:

\begin{theorem} \lb{t3.2} For any compact set, $\fre \subset \bbC$,
\begin{equation}
|T_n(z)|^{1/n} \to C(\fre) \exp(G_\fre(z)) = \exp(-\Phi_\fre(z))
\end{equation}
uniformly on compact subsets of $\bbC \setminus \mathrm{cvh}(\fre)$.
\end{theorem}

\begin{remarks} 1. This theorem is not new.  It appears in Saff--Totik \cite{SaT} as Theorem 3.9 in the more general context of weighted Chebyshev polynomials.

\smallskip

2.  Our proof is different.  They first control the density of zeros and use that to prove this result; shortly, we'll go in the other direction.
\end{remarks}

\begin{proof} Recall the Bernstein--Walsh Lemma, which says that, for any polynomial, $P$, of degree $n$, and any compact set $\fre \subset \bbC$, and so for $T_n$, one has that for all $z \in \bbC$
\begin{equation} \lb{3.2}
|T_n(z)| \le \norm{T_n}_\fre \exp(n G_\fre(z))
\end{equation}
Taking $n$th roots and using $C(\fre) = \exp(-R(\fre))$, we get that
\begin{equation} \lb{3.3}
|T_n(z)|^{1/n} \le Y(n) \exp(-\Phi_\fre(z)), \qquad Y(n) \equiv \norm{T_n}_\fre^{1/n}/C(\fre)
\end{equation}

Fej\'{e}r's theorem \cite{Fej} says that all the zeros of $T_n$ lie in $\mathrm{cvh}(\fre)$ so on $\widetilde{\Omega} \equiv (\bbC \cup \{\infty\}) \setminus \mathrm{cvh}(\fre)$, we have that
\begin{equation} \lb{3.4}
h_n(z) \equiv \log(Y(n)) - \Phi_\fre(z) - \tfrac{1}{n} \log(|T_n(z)|)
\end{equation}
are non-negative harmonic functions with $h_n(\infty) =  \log(Y(n))$.  By the FFS Theorem, $Y(n) \to 1$, so by Harnack's inequality, $h_n$ goes to zero, uniformly on compact subsets of $\widetilde{\Omega}$.
\end{proof}

Standard methods of going from root asymptotics to control on the density of zeros (see, for example \cite{ST, EqMC}) imply the following result (which is a special case of Theorem 4.7 of \cite{SaT}):

\begin{theorem} \lb{t3.3} Let $\fre \subset \bbR$ be compact.  Then the density of zeros measures for $T_n$ converge to the equilibrium measure for $\fre$.
\end{theorem}

\section{Totik--Widom Bounds} \lb{s4}

\begin{proof}[Proof of Theorem~\ref{t1.4}] Consider the function
\begin{equation} \lb{4.1}
h(z) \equiv G_\fre(z) - G_{\fre_n}(z)
\end{equation}
This function is harmonic on $(\bbC \cup \{\infty \})\setminus \fre_n$.  $h$ is harmonic at $\infty$ since the $\log(|z|)$ terms cancel and one can use the removable singularities theorem.  One has that
\begin{equation} \lb{4.2}
h(\infty) = R(\fre) - R(\fre_n) = \log \left[ \frac{C(\fre_n)}{C(\fre)} \right]
\end{equation}

Since $d\rho_n$ is harmonic measure, \eqref{1.7} holds for $h$.  Since $h(x) = G_\fre(x)$ on $\fre_n$, if $\{K_j\}_{j=1}^M$ are the gaps for $\fre$, then, using \eqref{2.5},
\begin{equation} \lb{4.3}
h(\infty) \le \sum_{j=1}^M \rho_n(K_j)  \, \max_{x \in K_j} (G_\fre(x)) \le \frac{1}{n} \sum_{j=1}^M G_\fre(w_j)
\end{equation}
Since regularity of $\fre$ implies that $G_\fre$ vanishes at the ends of each gap, the maximum is taken at a critical point, $w_j$, and the sum is precisely the Parreau--Widom sum. Exponentiating and using $\norm{T_n}_\fre \le 2C(\fre_n)^n$ and \eqref{4.2}, we get the result, \eqref{1.16}.  \end{proof}

\begin{remarks} 1. Because of the final assertion in Theorem \ref{t2.4}, one can replace $PW(\fre)$ in \eqref{1.16} by $\tfrac{1}{2}PW(\fre) + \tfrac{1}{2} S_n$ where $S_n$ is the sum of the $n$ largest values among the $G_\fre(w_j)$.

\smallskip

2. Just as the FFS Theorem is equivalent to information about $C(\fre_n)$, so \eqref{1.14} is equivalent to
\begin{equation} \lb{4.4}
C(\fre_n) \le C(\fre) \left(1+\frac{q}{n} \right)
\end{equation}
for some $q$, which is the form that Totik proved in the finite gap case.  To see that \eqref{4.4} $\Rightarrow$ \eqref{1.14}, we note that it is well known that $(1+\tfrac{q}{n})^n$ is monotone increasing in $n$ to $\exp(q)$, so given \eqref{2.3}, \eqref{4.4} implies \eqref{1.14} with $Q = 2 \exp(q)$.  In the other direction, since
\begin{equation}
e^x - 1 = \int_0^x e^y dy \le x e^x \; \mbox{ for $x \ge 0$}
\end{equation}
we have for $n \ge 1$ and $\widetilde{Q} \ge 1$ that
\begin{equation}
\widetilde{Q}^{1/n} = e^{\log(\widetilde{Q})/n} \le 1 + \frac{\log(\widetilde{Q})}{n} e^{\log(\widetilde{Q})/n} \le 1+\frac{q}{n}
\end{equation}
where $q = \widetilde{Q} \log(\widetilde{Q})$.  This shows that given \eqref{2.3}, \eqref{1.14} implies \eqref{4.4} with $\widetilde{Q} = Q/2$ and $q$ as just given.

\smallskip

3.  This seems to be the first example of an upper bound of the form \eqref{1.14} for an $\fre$ with an infinite number of components although there have been a number of papers, as we noted, for fairly general finite component sets in $\bbC$.  Since $T_n$ is defined variationally, in principle, upper bounds shouldn't be hard -- one need only guess a clever trial polynomial and indeed, that's what the earlier work does.  Our approach uses potential theory and doesn't seem to involve a variational guess although, it might be argued that underlying \eqref{2.3} is using the discriminant of $\fre$ as a trial polynomial.  But that's of course $T_n$ which is not merely a trial polynomial!
\end{remarks}

\section{Szeg\H{o}--Widom Asymptotics} \lb{s5}

In this section, we will prove Theorem \ref{t1.9} settling Widom's 1969 conjecture. We begin with some notation and some preliminaries. $\fre \subset \bbR$ will be a finite gap set (although some results like Theorem \ref{t5.1} hold more generally).  As before,
\begin{equation}
\Omega \equiv (\bbC \cup \{\infty \})\setminus \fre, \quad G \equiv G_\fre, \quad B \equiv B_\fre
\end{equation}
$\fre_n$ is given by \eqref{1.3}, $G_n$ is its Green's function, $B_n \equiv B_{\fre_n}$, and $\Omega_n \equiv (\bbC \cup \{\infty \})\setminus \fre_n$.  We'll let
\begin{equation}
\wti{\Omega} \equiv (\bbC \cup \{\infty \})\setminus \textrm{cvh}(\fre),
\quad \wti{G} \equiv G_{\textrm{cvh}(\fre)},
\quad \wti{B} \equiv B_{\textrm{cvh}(\fre)}
\end{equation}
Since Green's functions of regular sets decrease as the set increases (by an application of the maximum principle), we have that for all $z \in \bbC$
\begin{eqnarray} \lb{5.1} \nonumber
\fre \subset \fre_n \subset \textrm{cvh}(\fre) & \Rightarrow & \wti{G}(z) \le G_n(z) \le G(z) \\
& \Rightarrow & |B(z)| \le |B_n(z)| \le |\wti{B} (z)|
\end{eqnarray}

Define
\begin{equation} \lb{5.2}
L_n(z) \equiv \frac{T_n(z) B(z)^n}{C(\fre)^n}
\end{equation}
so that \eqref{1.20} says that $\lim_{n \to \infty} (L_n(z) - F_n(z)) = 0$ uniformly on compact subsets of the universal cover of $\Omega$.  By the Bernstein--Walsh lemma, \eqref{3.2}, and \eqref{1.14}, for any $n$ and $z$,
\begin{equation}
|L_n(z)| \le \norm{T_n}_\fre C(\fre)^{-n} \le Q
\end{equation}
Thus, by the Vitali convergence theorem, it suffices to prove that $\lim_{n \to \infty} (L_n(z) - F_n(z)) = 0$ uniformly on compact subsets of $\wti{\Omega}$.

By \eqref{2.4} and \eqref{2.3}, we have that
\begin{align} \nonumber \lb{5.3}
L_n(z) & =  \tfrac{1}{2} \norm{T_n}_\fre (B_n(z)^{-n}+B_n(z)^n)B(z)^n C(\fre)^{-n} \\
& = (1+B_n(z)^{2n}) \frac{C(\fre_n)^n B(z)^n} {C(\fre)^n B_n(z)^n} \\ \nonumber
& = (1+B_n(z)^{2n}) H_n(z)
\end{align}
where
\begin{equation} \lb{5.4}
H_n(z) \equiv \frac{C(\fre_n)^n B(z)^n} {C(\fre)^n B_n(z)^n}
\end{equation}

Since, by \eqref{5.1}, $|B_n(z)| \le |\wti{B}(z)|$ and $|\wti{B}(z)| < 1$ on $\wti{\Omega}$, we conclude that to prove \eqref{1.20}, it suffices to prove that uniformly on compact subsets of $\wti{\Omega}$,
\begin{equation} \lb{5.5}
\lim_{n \to \infty} (H_n(z) - F_n(z)) = 0
\end{equation}

Lest it go by too fast, we want to note that, in the above, we canceled two factors of $2$ -- namely those in \eqref{2.3} and \eqref{2.4} -- and it is this cancelation that enables the proof to work.  For an additional use of one of the factors of $2$ allows us to rewrite \eqref{1.22} as
\begin{equation}
\lim_{n \to \infty} \frac{C(\fre_n)^n}{C(\fre)^n \norm{F_n}_\Omega} = 1
\end{equation}
Since the first inequality in \eqref{5.1} implies that
\begin{equation} \lb{5.6}
\norm{H_n}_{\wti{\Omega}} \le \frac{C(\fre_n)^n}{C(\fre)^n} \le \frac{Q}{2}
\end{equation}
we see that \eqref{1.22} and \eqref{2.3} imply that
\begin{equation} \lb{5.7}
\limsup_{n \to \infty} \frac{\norm{H_n}_{\wti{\Omega}}}{\norm{F_n}_\Omega} \le 1
\end{equation}
It can be regarded that our ability to settle Widom's conjecture is just what is in this paragraph and the use of various arguments already in his paper.

Since there is a uniform upper bound on $H_n$ and $F_n$, Montel's theorem implies compactness, so we need only show convergence of enough subsequences and we can pick them so that $H_{n(j)}$ has a limit which we'll show is a trial function for Widom's variational problem -- \eqref{5.7} will then imply the limit must be the limit of the $F_{n(j)}$.  For this to work, we need to consider $H_n$ on a larger region than $\wti{\Omega}$.  $H_n$ is, of course, defined as a multivalued function on $\Omega_n$ so we'll need the following to control $\fre_n$ for $n$ large.  Recall (see point (b) after the alternation theorem) that each gap, $K$, of $\fre$ has at most one zero, $\zeta_n^{(K)}$, of $T_n$.

\begin{theorem} \lb{t5.1} Let $K = (r,s)$ be a gap of $\fre$.
\begin{SL}
\item[{\rm{(a)}}] If, for some subsequence, $\{n(j)\}_{j=1}^\infty$,  $\zeta_{n(j)}^{(K)}$ has a limit $\zeta_\infty^{(K)} \in K$, then for large $j$, $\fre_{n(j)} \cap K$ is a closed interval containing $\zeta_{n(j)}^{(K)}$ of size bounded by $e^{-Dn(j)}$ for some $D>0$.

\item[{\rm{(b)}}] If, for some subsequence, $\{n(j)\}_{j=1}^\infty$, $K \setminus \fre_{n(j)}$ is connected, then for some $C>0$,
\begin{equation} \lb{5.7a}
|K \cap \fre_{n(j)}| \le C n(j)^{-2}
\end{equation}

\item[{\rm{(c)}}] If, for some subsequence, $\{n(j)\}_{j=1}^\infty$, $\zeta_{n(j)}^{(K)}$ has a limit which is $r$ or $s$ and if $K \setminus \fre_{n(j)}$ is not connected, then for some $C>0$,
\begin{equation} \lb{5.7b}
|K \cap \fre_{n(j)}| \le C n(j)^{-1}
\end{equation}
and that intersection approaches r or s.
\end{SL}
\end{theorem}

\begin{remarks} 1. For our application here, all we need is that in the first case, the band shrinks to a point and, in the last two cases, the band moves to the edges.  But the quantitative estimates are not hard, are interesting and we think optimal (as to order) in the first two cases and perhaps also in the third.

\smallskip

2. The second case can occur if the band of $\fre_n$ that intersects $K$ has a piece in $\fre$ (which is always the case if there are no zeros of $T_n$ in $K$) or if the band is entirely in $[r,s]$ but one of the touching gaps is closed, i.e., one edge is $r$ or $s$.  It is also the case if $K \cap \fre_n = \emptyset$.

\end{remarks}

\begin{lemma} \lb{t5.2} If, for some subsequence, $\{n(j)\}_{j=1}^\infty$, $\zeta_{n(j)}^{(K)}$ has a limit $\zeta_\infty^{(K)} \in K$, then for some $\delta >0$ and all large $j$, we have that for some $D>0$ and all $x \in [\zeta_\infty^{(K)} - \delta, \zeta_\infty^{(K)} + \delta]$,
\begin{equation} \lb{5.8}
\frac{T_{n(j)}(x)}{\norm{T_{n(j)}}_\fre} = \left(x - \zeta_{n(j)}^{(K)}\right) Q_j(x), \qquad |Q_j(x)| > e^{Dn(j)}
\end{equation}
\end{lemma}

\begin{proof}
The first part of \eqref{5.8} holds where $Q_j$ is the product of $x-x^{(n)}_j$ over all zeros other than $\zeta_{n(j)}^{(K)}$ divided by $\norm{T_{n(j)}}_\fre$.  By Theorems \ref{t1.2} and \ref{t3.3}, for $\delta$ small, $\lim_{j \to \infty} n(j)^{-1} \log \vert Q_j(x)\vert = G(x)$ uniformly on $[\zeta_\infty^{(K)} - \delta, \zeta_\infty^{(K)} + \delta]$.  Since $G$ is bounded away from $0$ on this interval, the second part of \eqref{5.8} is valid.  \end{proof}

Recall that we use $\Delta_n$ for ${2 T_n}/{\norm{T_n}_\fre}$.

\begin{lemma} \lb{t5.2a} 
Let $v \in K$ be the unique critical point of $G$ in $K$.  Suppose that for some subsequence, $\{n(j)\}_{j=1}^\infty$, and $\delta > 0$, we have that $\fre_{n(j)} \cap (v-\delta, v+ \delta) = \emptyset$.  Then for $j$ large, there is a single zero, $c_{n(j)}$, of $\Delta_{n(j)}'$ in $(v- \delta, v + \delta)$ and %
\begin{equation} \lb{5.8a}
\lim_{j \to \infty} c_{n(j)} = v
\end{equation}
\end{lemma}

\begin{proof} We have that uniformly in $[v-{\delta}/{2}, v+ {\delta}/{2}]$, $\partial G_{n(j)} \to \partial G$. By \eqref{2.8}, zeros of $\Delta_{n(j)}'$ are precisely the critical points of $G_{n(j)}$.  This implies uniqueness of the zero while the uniform convergence implies existence and convergence. \end{proof}

It will be useful to have notation for the connected components of
\begin{equation}
\fre =\bigcup_{k=1}^{p}[a_k,b_k], \quad a_k < b_k < a_{k+1} < b_{k+1}, \quad k=1,\dots,p-1
\end{equation}
It will also be convenient to have a notation for connected components of $\fre_n$. We will denote these by
\begin{equation}
[a_{n,k},b_{n,k}], \quad k=1,\dots,s_n, \quad 1\leq s_n < 2p
\end{equation}
since there is at most one extra band or partial band in each gap. Then
$\fre_n=\bigcup_{k=1}^{s_n}[a_{n,k},b_{n,k}]$
is a disjoint union.  We let $\{c_{n,k}\}_{k=1}^{s_n-1}$ be the zeros of $\Delta_n'(x)$ not in $\fre_n$, labeled so that $a_{n,k}<b_{n,k}<c_{n,k}<a_{n,k+1}<b_{n,k+1}$. 
By \eqref{2.13a}, we have that
\begin{equation} \lb{5.8b}
w_n(x) = \frac{1}{\pi n}\frac{|\Delta_n'(x)|}{\sqrt{4-\Delta_n(x)^2}}
= \frac{\frac{1}{\pi}\prod_{k=1}^{s_n-1}|x-c_{n,k}|} {\prod_{k=1}^{s_n}|(x-a_{n,k})(x-b_{n,k})|^{1/2}}
\end{equation}

\smallskip

\begin{lemma} \lb{t5.2b}For $k=2,\dots,s_n-1$ and $x\in[a_{n,k},b_{n,k}]$,
\begin{equation} \lb{5.8c}
w_n(x) \geq \frac{1}{\pi} \frac{|x-c_{n,k}|}{|x-b_{p}|} \frac{|x-c_{n,k-1}|}{|x-a_{1}|} \frac{1}{|(x-a_{n,k})(x-b_{n,k})|^{1/2}}
\end{equation}
When $x\in[a_{n,1},b_{n,1}]$,
\begin{equation} \lb{5.8d}
w_n(x) \geq \frac{1}{\pi} \frac{|x-c_{n,1}|}{|x-b_{p}|} \frac{1}{|(x-a_{n,1})(x-b_{n,1})|^{1/2}}
\end{equation}
When $x\in[a_{n,{s_n}},b_{n,{s_n}}]$,
\begin{equation} \lb{5.8e}
w_n(x) \geq \frac{1}{\pi} \frac{|x-c_{n,{s_n}}|}{|x-a_{1}|} \frac{1}{|(x-a_{n,{s_n}})(x-b_{n,{s_n}})|^{1/2}}
\end{equation}
\end{lemma}

\begin{proof} Suppose first that $k = 2, \dots, s_n-1$. Since $a_{n,j}<b_{n,j}<c_{n,j}<a_{n,j+1}<b_{n,j+1}$, $j=1,\dots,s_n-1$, the following estimates hold,
\begin{align}
\frac{|x-c_{n,j}|}{|(x-a_{n,j})(x-b_{n,j})|^{1/2}} &\geq 1 \;\text{ for all } x<a_{n,j}\\
\frac{|x-c_{n,j-1}|}{|(x-a_{n,j})(x-b_{n,j})|^{1/2}} &\geq 1 \; \text{ for all } x>b_{n,j}
\end{align}
Thus, we get a lower bound if we drop the first zero and second band, second zero and third band, $\dots$ and similarly drop the last zero and next to last band, $\dots$.  We can then use, since $x > b_1$, that $|x - b_1| < |x - a_1|$ and similarly on the other end to get the lower bound \eqref{5.8c}.

For the cases $k = 1$ or $s_n$, we only need to do the zero shielding on one side and we get \eqref{5.8d} and \eqref{5.8e}. \end{proof}

\begin{proof}[Proof of Theorem~\ref{t5.1}] \rm{(a)} By Lemma \ref{t5.2}, $T_{n(j)}(x) = \pm \norm{T_{n(j)}}_\fre$ has solutions within $e^{-Dn(j)}$ of $\zeta_{n(j)}^{(K)}$, so there is a band of $\fre_{n(j)}$ of size less than $2e^{-Dn(j)}$ containing that zero. Since $K \cap \fre_{n(j)}$ contains at most one band, we have the claimed result.

\rm{(b)}, \rm{(c)} (common part) We begin by showing that when the zeros don't have a limit point in $(r,s)$, then any part of $K \cap \fre_{n(j)}$ approaches the edges of $K$.  If there is a point in $(r+2\epsilon, s-2\epsilon) \cap \fre_{n(j)}$ and a zero in $(r,r+\epsilon)$, then since $\fre_{n(j)} \cap K$ is connected, we have $[r+\epsilon, r+2 \epsilon] \equiv I_- \subset \fre_{n(j)} \cap K$ and similarly with $I_+ \equiv [s-2\epsilon, s-\epsilon]$ if the zero is within $\epsilon$ of $s$.  If there is no zero in $K$, then the band extends past one of the edges of $K$ and again either $I_-$ or $I_+$ is in $\fre_{n(j)} \cap K$.  Thus, if it is not eventually true that $(r+2\epsilon, s-2\epsilon) \cap \fre_{n(j)} = \emptyset$, then infinitely often, either $C(\fre_n) \ge C(\fre \cup I_-)$ or the same for $I_+$.  This is inconsistent with Theorems \ref{2.9} and \ref{3.1} which imply that $\lim_{n \to \infty} C(\fre_n) = C(\fre)$ proving, by contradiction, the desired result that the bands in $\fre$ approach the edges.

\rm{(b)} We consider the case where the band spreads into the left side of $K$ and $k=2,\dots,s_n-1$. The argument is similar for $k=1$ or $k=s_n$ or at the other ends of the gaps.  We let $[a, b]$ be the connected component of $\fre$ with $b=r$.  We thus suppose that $K \cap \fre_{n(j)} = [\alpha_j, \beta_j]$ with $\alpha_j = b$ and $\beta_j \to b$ and that $b_{n(j),k(j)} = \beta_j$ and $a_{n(j),k(j)} \le a$. We assume that $c_{n(j)}$ and $v$ are as in Lemma \ref{t5.2a} and that \eqref{5.8a} holds.  Thus, we can suppose that $j$ is so large that for $x < \beta_j$ we have that $|x - c_{n(j)}| > \tfrac{v-b}{2}$.  \eqref{5.8c} then becomes
\begin{equation} \lb{5.8f}
w_{n(j)}(x) \geq \frac{v - b}{2 \pi} \frac{|b-a|^{1/2}}{|b_{p}-a_{1}|^2} \frac{1}{|x-\beta_j|^{1/2}}
\end{equation}
Integrating from $\alpha_j$ to $\beta_j$ and using Theorem \ref{t2.4}, we get that
\begin{equation} \lb{5.8g}
\frac{1}{n(j)} \geq \int_{\alpha_j}^{\beta_j} w_{n(j)}(x)dx \geq \frac{v-b}{\pi} \frac{|b-a|^{1/2}}{|b_{p}-a_{1}|^2} |\beta_j-\alpha_j|^{1/2}
\end{equation}
proving \eqref{5.7a}.

\rm{(c)} The argument is similar to \rm{(b)} except that now we don't have $c_{n(j),k(j)-1} < a < b = \alpha_j$ but only $b < c_{n(j),k(j)-1} < \alpha_j = a_{n(j),k(j)}$ so \eqref{5.8f} is replaced by
\begin{equation} \lb{5.8h}
w_{n(j)}(x) \geq \frac{v - b}{2 \pi} \frac{|x-\alpha_j|^{1/2}}{|b_{p}-a_{1}|^2} \frac{1}{|x-\beta_j|^{1/2}}
\end{equation}
which leads to \eqref{5.7b} using $\int_0^1 \sqrt{\frac{x}{1-x}} < \infty$ and the analog of \eqref{5.8g}.
\end{proof}

\begin{proof}[Proof of Theorem~\ref{t1.9}] As noted, by boundedness and Montel's theorem, it suffices to prove that any subsequence has a subsubsequence for which \eqref{5.5} holds.  We can choose this subsubsequence, $\{n(j)\}_{j=1}^\infty$, so that:

\smallskip

\begin{SL}
\item[{\rm{(1)}}] The characters $\chi_{n(j)} \to \chi_\infty$ for some character $\chi_\infty$.  This implies, by Theorem \ref{t1.6}, that $F_{n(j)} \to F_\infty$ uniformly on compact subsets and $\norm{F_{n(j)}}_\Omega \to \norm{F_\infty}_\Omega$.

\smallskip

\item[{\rm{(2)}}] In each gap, $K_\ell$, of $\fre$, either $T_{n(j)}$ has a zero for $j$ large and the limit of the zeros is $x_\ell \in K_\ell$ or any zero in the gap, $K_\ell$, approaches $\fre$ in the limit or there is no zero in that gap.

\smallskip

\item[{\rm{(3)}}]  On $\widetilde{\Omega}$, the subsequence $H_{n(j)}(z)$ has a limit $H_\infty(z)$ by Montel's theorem.  In this case, it is easy to see that $\norm{H_\infty}_{\wti{\Omega}} \le \liminf \norm{H_{n(j)}}_{\wti{\Omega}}$.

\end{SL}

\smallskip

Let $\calL$ be  the set of gaps, $K_\ell$, with a limit point of zeros. The $H_{n(j)}$ can be continued along any curve in $\Omega_{n(j)}$.  Since all the harmonic measures of sets in $\fre_n$ are multiples of ${1}/{n}$, $B_n^n$ is analytic on $\Omega_n$.  It follows that $H_{n(j)}$ are defined and character automorphic with character $\chi_{n(j)}$ on sets which converge to the universal cover of $\Omega$ with the points that lie over the set $\{x_\ell\}_{\{K_\ell \in \calL\}}$ removed. Thus, by \eqref{5.6} and Vitali's theorem, $H_\infty$ has a continuation to that set. By \eqref{5.6} again, $x_\ell$ are removable singularities.

Since near $\infty$,  for any $\frg$, $B_{\frg}(z) = {z}/{C(\frg)} + \rm{O}(1)$, we see that for each $n$, $H_n(\infty) =1 \Rightarrow H_\infty(\infty) =1$. Thus, $H_\infty$ is a trial function for the problem where $F_\infty$ is the minimizer. By \eqref{5.7} and continuity of $\norm{F_n}_\Omega$, we see that $\norm{H_\infty}_\Omega \le \norm{F_\infty}_\Omega$.  Thus, the uniqueness of the minimizer implies that $H_\infty = F_\infty$, proving the desired convergence on $\widetilde{\Omega}$. \end{proof}


\end{document}